\documentclass[11pt]{amsart}

\usepackage{amsmath,amsfonts,amsthm,amsopn,amssymb}
\usepackage{graphicx}
\usepackage{enumerate}
\usepackage{epsfig,verbatim}
\usepackage{mathrsfs}
\usepackage{hyperref}
\usepackage{bbm}

\newtheoremstyle{mythm}{6pt}{6pt}{\itshape}{}{\bfseries}{.}{ }{} 
\newtheoremstyle{mydef}{6pt}{6pt}{}{}{\bfseries}{.}{ }{}         
\newtheoremstyle{myrem}{6pt}{6pt}{}{}{\scshape}{.}{ }{}          

\theoremstyle{mydef}
\newtheorem{definition}{Definition}[section]

\theoremstyle{mythm}
\newtheorem{theorem}[definition]{Theorem}

\newtheorem{lemma}[definition]{Lemma}

\theoremstyle{myrem}
\newtheorem{remark}[definition]{Remark}

\renewenvironment{proof}{\par\medbreak\noindent\textit{Proof}:\hskip.5em\ignorespaces}{\hfill\qedsymbol\medbreak} 
\renewcommand{\qedsymbol}{\hfill\rule{2mm}{2mm}}                                                                  

\newcommand{\mfrac}[2]{\text{\footnotesize{$\dfrac{#1}{#2}$}}}
\newcommand{\lfrac}[2]{\text{\small{$\dfrac{#1}{#2}$}}}

\newcommand{\cV}{\mathcal{V}}
\newcommand{\R}{\mathbb{R}}

\newcommand{\Rtd}{{\R^{2d}}}
\newcommand{\Rt}{{\R^{2}}}
\newcommand{\So}{\boldsymbol S_0}

\newcommand{\LtR}{\mathbf{L}^2(\R)}

\allowdisplaybreaks

\begin{document}
\title[Rel.~Samp.~STFT TF-Loc.~Func.]{Relevant sampling of the short-time Fourier transform of time-frequency localized functions}

\author[G.~Velasco]{Gino Angelo Velasco}
\address{Institute of Mathematics, University of the Philippines, Diliman, Quezon City 1101, Philippines}
\email[Gino Angelo Velasco]{gamvelasco@math.upd.edu.ph}

\thanks{The author acknowledges the Office of the Chancellor of the University of the Philippines Diliman, through the Office of the Vice Chancellor for Research and Development, for funding support through the Ph.D.~Incentive Award.}


\begin{abstract}
We study the random sampling of the short-time Fourier transform of functions that are localized in a compact region in the time-frequency plane. We follow the approach introduced by Bass and Gr\"{o}chenig for band-limited functions, and show that with a high, controllable probability, a sufficiently dense set of local random samples from the region of concentration in the time-frequency plane yields a sampling inequality for the short-time Fourier transform of time-frequency localized functions on the region. 
\end{abstract}

\maketitle

\section{Introduction}

The short-time Fourier transform (STFT) is a standard tool used in the analysis and processing of signals. The STFT of a function or signal $f$ can be interpreted as the inner product of $f$ with a time-frequency shifted version of a single window function $\varphi$, and the reconstruction of $f$ from its STFT is possible via an inversion formula. However, this continuous representation is highly redundant, and to lessen the redundancy, a sampling of the STFT is done. Most of the studies done on the sampling of the STFT of functions are in the context of frame theory. In particular, the sampling of the STFT corresponds to having a discrete set of time-frequency shifts of $\varphi$, and sampling inequalities for the STFT translate to the time-frequency shifts of $\varphi$ forming a so-called Gabor frame. Of particular interest is in finding the conditions on the window function and the point set $\Lambda$ for which the corresponding set of time-frequency shifts of $\varphi$ via $\Lambda$ form a frame. Many results have appeared concerning irregular Gabor frames, or irregular sampling of the STFT, e.g.~\cite{fekost95,rast95,chdehe99,suzh02,cach03,fesu07}, to name a few, and most amount to density conditions on $\Lambda$.

In this paper, we study the \emph{relevant sampling} of the STFT of a function, where we establish a sampling inequality from random sampling points lying only on a compact subset $\Omega$ of $\mathbb{R}^2$, since one faces the problem of finding the appropriate probability distribution for the sampling points if an unbounded and non-compact set like $\mathbb{R}^2$ is considered. The notion of relevant sampling was introduced by Bass and Gr\"{o}chenig \cite{bagr10,bagr13}, where they investigated the probability that from random local sampling points in a compact set, a sampling inequality holds for functions that are band-limited but are essentially supported on the compact set. F\"{u}hr and Xian \cite{fuxi14} extended the results to the more general setting of finitely generated shift-invariant spaces. 

We apply the relevant sampling approach to the STFT of functions that satisfy some locality property in $\Omega$. While our results are mostly analogs of those in \cite{bagr10,bagr13,fuxi14}, the relevant sampling of the STFT provides an interesting supplement to the existing results in the band-limited setting and the setting of finitely generated shift-invariant spaces, and gives some new insights in the study of irregular Gabor frames.

The paper is organized as follows. In the next section, we recall some tools from time-frequency analysis, namely, the short-time Fourier transform and Gabor systems and frames. In Section \ref{sec:TFConcSTFT}, we mention some properties of time-frequency localization operators, and prove some inequalities involving time-frequency localized functions and functions on subspaces of eigenfunctions of time-frequency localization operators. We present the relevant sampling results for the STFT of time-frequency localized functions in Section \ref{sec:RelSampSTFT}. As in \cite{bagr10,bagr13,fuxi14}, we first establish a random sampling inequality for functions in a subspace of eigenfunctions of the time-frequency localization operator, and use this to obtain the sampling inequality for time-frequency localized functions. Finally, in Section \ref{sec:ApproxRec}, we look at an approximate reconstruction of a time-frequency localized function from the local STFT samples and provide numerical examples.

\section{Preliminaries}

In this section we recall some definitions and properties about the short-time Fourier transform and Gabor frames. For a detailed introduction to time-frequency analysis, we refer the reader to \cite{gr01}. 

\subsection{Short-Time Fourier transform}

The \emph{short-time Fourier transform} (STFT) of $f\in \mathbf{L}^2(\mathbb{R})$ with respect to $\varphi$ is given by 
\begin{equation*}
\cV_{\varphi}f(z) = \cV_{\varphi}f(x,\omega) = \int_{\R}f(t)\,\overline{\varphi(t-x)}\,e^{-2\pi i\omega\cdot t}dt=\langle f,\pi(z)\varphi\rangle,
\end{equation*}
where $z=(x,\omega)\in\Rt$ and $\pi(z)$ is the time-frequency shift operator given by $\pi(z)f = f(t-x)\,e^{2\pi i\omega\cdot t}$. The STFT is an isometry from $\mathbf{L}^2(\R)$ to $\mathbf{L}^2(\Rt)$, i.e. $\|\cV_{\varphi}f\|_2 = \|\varphi\|_2\|f\|_2$, and inversion is realized using the formula
 \begin{equation}\label{eq:InvForm}
  f = \cV_{\varphi}^{\ast}\,\cV_{\varphi}f = \iint_{\Rt} \cV_{\varphi}f(z)\pi(z)\varphi\,dz,
 \end{equation}
where the vector-valued integral above and similar expressions in the sequel are understood in a weak sense, cf.~\cite[Section 3.2]{gr01}.
 
The membership of the STFT in $\mathbf{L}^p(\Rt)$ provides a definition of a class of function spaces called \emph{modulation spaces}. For a historical account of their development and role in time-frequency analysis, see \cite{fe06}. Let $\varphi_0$ be the Gaussian, i.e.~$\varphi_0(t) = e^{-\pi t^2}$. The modulation space $\So(\R)$, also known as Feichtinger's algebra, is the space of all  $f\in \mathbf{L}^2(\mathbb{R})$ such that $\|f\|_{\boldsymbol{S}_0}:=\|\cV_{\varphi_0}f\|_{\mathbf{L}^1(\Rt)}<\infty$. It is the smallest Banach space isometrically invariant under time-frequency shifts and the Fourier transform, and is continuously embedded in $\mathbf{L}^2(\R)$ and $\mathbf{L}^1(\R)$, cf.~\cite{fezi98}. Some conditions for membership of $f$ in $\So(\R)$ include $f$ being band-limited and belonging to $\mathbf{L}^1(\R)$, or both $fw_s$ and $\hat{f}w_s$ belonging to $\mathbf{L}^2(\R)$, where $\hat{f}$ is the Fourier transform of $f$ and $w_s(t) = (1+t^2)^{s/2},\,s>1$.

\subsection{Gabor systems and frames}

Given a window function $\varphi\in\mathbf{L}^2(\R)$ and a countable point set $\Lambda\in\mathbb{R}^2$, the \emph{Gabor system} $\mathcal{G}(\varphi,\Lambda)$ is given by $\mathcal{G}(g,\Lambda) = \{\pi(\lambda)\varphi\,:\,\lambda\in\Lambda\}$. We say that $\mathcal{G}(\varphi,\Lambda)$ is a \emph{Gabor frame} if there exist positive constants $A,B>0$ such that for all $f\in\mathbf{L}^2(\R)$ 
\begin{equation*}
 A\|f\|_2^2\leq\sum_{\lambda\in \Lambda}|\langle f,\pi(\lambda)\varphi\rangle|^2\leq B\|f\|_2^2.
\end{equation*}
Since $\cV_{\varphi}f(z) = \langle f,\pi(z)\varphi\rangle$, the above inequalities are also interpreted as sampling inequalities on the STFT. A result that gives an upper inequality was proved in \cite{ba05-1}. Here, we let $Q_h(v) = [v_1-h/2,v_1+1/2]\times[v_2-h/2,v_2+h/2]$, where $h>0,\,v = (v_1,v_2)\in\mathbb{R}^2$

\begin{theorem}{\cite[Theorem 2.2.6]{ba05-1}}\label{thm:BesselBd}
  Let $\varphi\in\So(\mathbb{R})$ and let $\Lambda$ be a relatively separated set in $\Rt$, i.e.~for some (and therefore every) $h>0$, there is $N_h(\Lambda)\in\mathbb{N}$ such that $\sup_{m\in\mathbb{Z}^2}\#(\Lambda\cap Q_h(hm))\leq N_h(\Lambda)$. Then there exists $B>0$ such that for all $f\in\LtR$,
  \begin{equation*}
    \sum_{\lambda\in\Lambda}|\cV_{\varphi}f(\lambda)|^2\leq B\,\|f\|_2^2.
  \end{equation*}
\end{theorem}

\begin{remark}
  From the proof of \cite[Theorem 2.2.6]{ba05-1}, a suitable choice for $B$ is 
  \begin{equation}\label{eq:BesselBd}
    B = N_h(\Lambda)C_{\varphi,\varphi_0}^2\|\cV_{\varphi_0}\varphi_0\|_{\boldsymbol{W}}\|\varphi\|_{\So}^2,
  \end{equation}
  where $C_{\varphi,\varphi_0} = \inf\Big\{C>0\,:\,\sum_{n\in\mathbb{N}}|b_n|\leq C\|\varphi\|_{\So},\,\varphi = \sum_{n\in\mathbb{N}}b_n\pi(z_n)\varphi_0,\,\sum_{n\in\mathbb{N}}|b_n|<\infty,\,\,z_n = (x_n,\omega_n)\in\mathbb{R}^2\Big\}$ and $\|\cV_{\varphi_0}\varphi_0\|_{\boldsymbol{W}} = \sum_{m\in\mathbb{Z}^2}\operatorname{ess\,sup}_{v\in Q_1(0)}|\cV_{\varphi_0}\varphi_0(v + m)|$.
\end{remark}

\section{Time-Frequency Concentration via the STFT}\label{sec:TFConcSTFT}
Time-frequency localization operators as introduced by Daubechies in \cite{da88} are built by restricting the integral in the inversion formula \eqref{eq:InvForm} to a subset of $\Rt$. Its properties, connections with other mathematical topics, and applications have been topics in various works, e.g. \cite{rato94,feno01,defeno02,cogr03,abdo12,doro13,doro14,dove16}.

Let $\Omega$ be a compact set in $\Rt$, $\chi_{\Omega}$ its characteristic (or indicator) function, and $\varphi$ a window function in $\mathbf{L}^2(\R)$, with $\|\varphi\|_2 = 1$. The \emph{time-frequency localization operator} $H_{\Omega,\varphi}$ is defined by
\begin{equation*}
   H_{\Omega,\varphi}f = \iint_{\Omega} \cV_{\varphi}f(z)\pi(z)\varphi\,dz = \cV_{\varphi}^{\ast}\,(\chi_{\Omega}\,\cV_{\varphi}f).
\end{equation*}

The above integral can be interpreted as the portion of the function $f$ that is essentially contained in $\Omega$. Moreover, the following inner product involving $H_{\Omega,\varphi}$ measures the function's energy inside $\Omega$:
\begin{equation}\label{eq:TFLocMeas}
  \langle  H_{\Omega,\varphi}f,f\rangle = \iint_{\Omega}\cV_{\varphi}f(z)\langle \pi(z)\varphi,f\rangle dz = \iint_{\Omega}|\cV_{\varphi}f(z)|^2 dz.
\end{equation}
We will say that a function $f\in \mathbf{L}^2(\R)$ is \emph{$(\varepsilon,\varphi)$-concentrated} inside $\Omega$ if  $\langle  H_{\Omega,\varphi}f,f\rangle\geq (1-\varepsilon)\|f\|_2^2$ or equivalently $\langle (I- H_{\Omega,\varphi})f,f\rangle\leq \varepsilon\|f\|_2^2$, where $I$ is the identity operator.

The time-frequency localization operator $H_{\Omega,\varphi}$ is a compact and self-adjoint operator so we can consider the spectral decomposition 
\begin{equation*}
   H_{\Omega,\varphi}f = \sum_{k = 1}^{\infty}\alpha_k\langle f,\psi_k\rangle \psi_k,
\end{equation*}
where $\{\alpha_k\}_{k=1}^{\infty}$ are the positive eigenvalues, with $\alpha_k\leq 1$ for all $k$, arranged in a non-increasing order and $\{\psi_k\}_{k = 1}^{\infty}$ are the corresponding eigenfunctions.
By the min-max theorem for compact, self-adjoint operators, the first eigenfunction has optimal time-frequency concentration inside $\Omega$ in the sense of \eqref{eq:TFLocMeas}, i.e.
\begin{equation*}
 \iint_{\Omega}|\cV_{\varphi}\psi_1(z)|^2 dz = \max_{\|f\|_2 = 1} \iint_{\Omega}|\cV_{\varphi}f(z)|^2 dz.
\end{equation*}

\bigskip

We denote by $\mathcal{P}_{V_N}$ the orthogonal projection operator to the subspace $V_N$ spanned by the eigenfunctions corresponding to the $N$ largest eigenvalues of $H_{\Omega,\varphi}$. The eigenfunctions $\{\psi_k\}_{k = 1}^{\infty}$ form an orthonormal subset of $\LtR$, possibly incomplete if the kernel $\operatorname{ker}(H_{\Omega,\varphi})$ of $H_{\Omega,\varphi}$ is nontrivial, and we write $f = \sum_{k = 1}^{\infty}\langle f,\psi_k\rangle\psi_k + f_{\operatorname{ker}}$, where $f_{\operatorname{ker}}$ is the orthogonal projection of $f$ onto $\operatorname{ker}(H_{\Omega,\varphi})$. We also have that $\langle H_{\Omega,\varphi}f,f\rangle  = \sum_{k = 1}^{\infty}\alpha_k|\langle f,\psi_k\rangle|^2$. We mention the standard estimate for the distribution of the eigenvalues of $\mathbf{H}_{\Omega,\varphi}$, which appears e.g.~in \cite{la75-1}. The version presented in \cite[Lemma 3.3]{abgrro16} is the following:
\begin{equation}\label{eq:EigEst}
  \Big|\#\{k\,:\,\alpha_k>1-\delta\}-|\Omega|\Big|\leq\max\Big\{\mfrac{1}{\delta},\mfrac{1}{1-\delta}\Big\}\left|\int_{\Omega}\int_{\Omega}|\cV_{\varphi}\varphi(z-z')|^2\,dz\,dz'-|\Omega|\right|.
\end{equation}

In the following lemma, we establish some inequalities involving the projection of a function on $V_N$. These are analogues to the case where the localization operator is via the composition of time- and band-limiting operators \cite{bagr13}, the case of shift-invariant space in \cite{fuxi14}, and the Dunkl setting in \cite{gh16-3}. 

\begin{lemma}\label{lem:TFlocApproxVN}
  Let $N\in\mathbb{N}$ and let $\gamma\in\mathbb{R}$ with $\alpha_N\geq\gamma\geq\alpha_{N+1}$. If $f$ is $(\varepsilon,\varphi)$-concentrated on $\Omega\subset\Rtd$, then 
  \begin{align}
    \|\mathcal{P}_{V_N}f\|_2^2&\geq \Big(1-\lfrac{\varepsilon}{1-\gamma}\Big)\|f\|_2^2, \label{lem:TFlocApproxVN-1}\\
    \Big\|f-\mathcal{P}_{V_N}f\Big\|_2^2&\leq \lfrac{\varepsilon}{1-\gamma}\|f\|_2^2,\,\text{ and } \label{lem:TFlocApproxVN-2}\\
    \langle  H_{\Omega,\varphi}\mathcal{P}_{V_N}f,\mathcal{P}_{V_N}f\rangle&\geq \gamma\Big(1-\lfrac{\varepsilon}{1-\gamma}\Big)\|f\|_2^2. \label{lem:TFlocApproxVN-3}
  \end{align}
\end{lemma}

\begin{proof} 
  Let $f$ be $(\varepsilon,\varphi)$-concentrated on $\Omega$, and without loss of generality, let $\|f\|_2 = 1$. Assume that $\|\mathcal{P}_{V_N}f\|_2^2 = \sum_{k = 1}^N|\langle f,\psi_k\rangle|^2 = K<1 - \lfrac{\varepsilon}{1-\gamma}$.
    Since $\|f\|_2^2 = 1 = \sum_{k = 1}^{\infty}|\langle f,\psi_k\rangle|^2 + \|f_{\operatorname{ker}}\|_2^2$ and $\alpha_1\geq\alpha_2\geq\cdots\geq\alpha_N\geq\gamma$, we have $\sum_{k = N+1}^{\infty}|\langle f,\psi_k\rangle|^2 = 1-K -\|f_{\operatorname{ker}}\|_2^2$ and $\sum_{k = N+1}^{\infty}\alpha_k|\langle f,\psi_k\rangle|^2 \leq \gamma(1-K -\|f_{\operatorname{ker}}\|_2^2).$ Moreover, since $\alpha_k\leq 1$ for each $k$, we obtain
    \begin{align*}
      \sum_{k = 1}^{\infty}\alpha_k|\langle f,\psi_k\rangle|^2 &\leq K + \gamma(1-K -\|f_{\operatorname{ker}}\|_2^2)\\
      &<(1-\gamma)\Big(1-\lfrac{\varepsilon}{1-\gamma}\Big) + \gamma - \gamma\|f_{\operatorname{ker}}\|_2^2\\
      &=1-\varepsilon -\gamma\|f_{\operatorname{ker}}\|_2^2<1-\varepsilon,
    \end{align*}
    where the second line follows from the assumption that $K<1 - \lfrac{\varepsilon}{1-\gamma}$. The resulting inequality above contradicts the hypothesis that $f$ is $(\varepsilon,\varphi)$-concentrated on $\Omega$ since $\langle H_{\Omega,\varphi}f,f\rangle  = \sum_{k = 1}^{\infty}\alpha_k|\langle f,\psi_k\rangle|^2$. Hence, $\|\mathcal{P}_{V_N}f\|_2^2\geq \Big(1-\lfrac{\varepsilon}{1-\gamma}\Big)\|f\|_2^2$.
    
    The inequality in \eqref{lem:TFlocApproxVN-2} follows from \eqref{lem:TFlocApproxVN-1} and the orthogonality of $\mathcal{P}_{V_N}f$ and $f-\mathcal{P}_{V_N}f$. Finally, to prove \eqref{lem:TFlocApproxVN-3}, since each $\psi_k$ is an eigenfunction of $H_{\Omega,\varphi}$ with corresponding eigenvalue $\alpha_k$, we have
	  \begin{equation*}
	   \langle  H_{\Omega,\varphi}\mathcal{P}_{V_N}f,\mathcal{P}_{V_N}f\rangle\geq \sum_{k = 1}^N\alpha_k|\langle f,\psi_k\rangle|^2\geq \gamma\|\mathcal{P}_{V_N}f\|_2^2,
	  \end{equation*}
	  and the conclusion follows from \eqref{lem:TFlocApproxVN-1}.
  \end{proof}

\section{Relevant Sampling of the STFT}\label{sec:RelSampSTFT}

We follow the approach in \cite{bagr13,fuxi14} for functions that are $(\varepsilon,\varphi)$-concentrated on a compact region $\Omega$ in the time-frequency plane, where $\varphi\in \LtR$ with $\|\varphi\|_2=1$. We apply a matrix Bernstein inequality due to Tropp \cite{tr12}. Let $\alpha_{\text{max}}(A)$ be the largest singular value of a matrix $A$ so that $\|A\| = \alpha_{\text{max}}(A^{\ast}A)^{1/2}$ is the operator norm.

  \begin{theorem}\cite[Theorem 1.4]{tr12}\label{thm:Tropp}
    Let $X_j$ be a finite sequence of independent, random, self-adjoint $N\times N$-matrices. Suppose that $\mathbb{E}(X_j) = 0$ and $\|X_j\|\leq B$ a.s. and let $\sigma^2 = \left\|\sum_{j = 1}^r\mathbb{E}(X_j^2)\right\|$. Then for all $t\geq 0$,
    \begin{equation*}
      \mathbb{P}\bigg(\alpha_{\operatorname{max}}\Big(\sum_{j=1}^r X_j \Big)\geq t\bigg)\leq N\exp\left(-\,\lfrac{t^2/2}{\sigma^2+Bt/3}\right).
    \end{equation*}
  \end{theorem}
  
  We take $X_j = T_j-\mathbb{E}(T_j)$ and obtain estimates for $\|X_j\|$, $\mathbb{E}(X_j^2)$, and $\|\sum_{j = 1}^r\mathbb{E}(X_j^2)\|$ in the next lemma.

  \begin{lemma}
    If $X_j = T_j-\mathbb{E}(T_j)$, then
    \begin{enumerate}[1.]
      \item[{\normalfont 1.}] $\|X_j\|\leq 1$,
      \item[{\normalfont 2.}] $\mathbb{E}(X_j^2)\leq \lfrac{1}{|\Omega|}\Delta$, and
      \item[{\normalfont 3.}] $\sigma^2 = \Big\|\displaystyle\sum_{j=1}^r\mathbb{E}(X_j^2)\Big\|\leq \lfrac{r}{|\Omega|}$.
    \end{enumerate}
  \end{lemma}
  \begin{proof}
    \begin{enumerate}[1.]
      \item The matrix norm of $X_j$ is estimated as follows:
	    \begin{align*}
	      \|X_j\| = \|T_j-\mathbb{E}(T_j)\| &= \sup\limits_{\|\mathbf{c}\| = 1}|\langle \mathbf{c},T_j \mathbf{c}\rangle_{\mathbb{C}^N} - \langle \mathbf{c},\mathbb{E}(T_j)\mathbf{c}\rangle_{\mathbb{C}^N}|\\
		&= \sup\limits_{\|f\|_2 = 1}\Big|\, |\mathcal{V}_{\varphi}f(\lambda_j)|^2-\mfrac{1}{|\Omega|}\langle  H_{\Omega,\varphi}f,f \rangle \Big|\\
		&\leq \|f\|_2^2\|\varphi\|_2^2 = 1
	    \end{align*}

      \item To find $\mathbb{E}(X_j^2)$, we use \eqref{eq:ExpTdiag} and obtain
	    \begin{align*}
	      \mathbb{E}(X_j^2) &= \mathbb{E}(T_j^2) - \mfrac{1}{|\Omega|}\mathbb{E}(T_j\Delta) - \mfrac{1}{|\Omega|}\mathbb{E}(\Delta T_j)+ \mfrac{1}{|\Omega|^2}\Delta^2\\
	      &= \mathbb{E}(T_j^2) - \mfrac{1}{|\Omega|}\mathbb{E}(T_j)\Delta - \mfrac{1}{|\Omega|}\Delta\mathbb{E}(T_j)+ \mfrac{1}{|\Omega|^2}\Delta^2\\
	      &= \mathbb{E}(T_j^2) - \mfrac{1}{|\Omega|^2}\Delta^2.
	    \end{align*}
	    Now we compare $T_j^2$ and $T_j$.
	    \begin{align*}
	      (T_j^2)_{km} &= \sum_{l = 1}^N(T_j)_{kl}(T_j)_{lm}\\
	      &= \sum_{k = 1}^N \overline{\langle \psi_k,\pi(\lambda_j)\varphi\rangle}\langle \psi_l,\pi(\lambda_j)\varphi\rangle\overline{\langle \psi_l,\pi(\lambda_j)\varphi\rangle} \langle \psi_m,\pi(\lambda_j)\varphi\rangle\\
	      &= \left(\sum_{l=1}^N|\langle \psi_l,\pi(\lambda_j)\varphi\rangle|^2\right)(T_j)_{km}\\
	      &= \|\mathcal{P}_{V_N}\varphi\|_2^2(T_j)_{km}\leq\|\varphi\|_2^2(T_j)_{km} = (T_j)_{km}
	    \end{align*}
	    We thus have $T_j^2\leq T_j$ and $\mathbb{E}(T_j^2)\leq\mathbb{E}(T_j) = \mfrac{1}{|\Omega|}\Delta$, so the expectation of $X_j^2$ gives
	    \begin{equation*}
		\mathbb{E}(X_j^2) = \mathbb{E}(T_j^2) - \mfrac{1}{|\Omega|^2}\Delta^2\leq \mfrac{1}{|\Omega|}\Delta.
	    \end{equation*}

    \item $\sigma^2 = \Big\|\displaystyle\sum_{j = 1}^r\mathbb{E}(X_j^2)\Big\|\leq 
      \mfrac{r}{|\Omega|}\|\Delta\|\leq \mfrac{r}{|\Omega|}$.
    \end{enumerate}    
  \end{proof}
  
We now provide a random sampling estimate for $V_N$.
  
\begin{theorem}\label{thm:ProbSampSubsp}
  Let $\Lambda_{\Omega}=\{\lambda_j\}_{j\in\mathbb{N}}$ be a sequence of independent  and identically distributed random variables that are uniformly distributed in
  $\Omega$. Then  for all $\nu\geq 0$ and $r\in\mathbb{N}$, we have
      \begin{small}
  \begin{equation*}\label{eq:PropProb}
      \mathbb{P}\left(\inf_{f\in V_N,\|f\|_2 = 1}\lfrac{1}{r}\sum_{j = 1}^r(|\mathcal{V}_{\varphi}f(\lambda_j)|^2-\lfrac{1}{|\Omega|}\langle  H_{\Omega,\varphi}f,f \rangle)\leq -\lfrac{\nu}{|\Omega|}\right)
      \leq  N\exp\left(-\lfrac{\nu^2 r}{|\Omega|(1+\nu/3)}\right).
  \end{equation*}
      \end{small}
\end{theorem}

\begin{proof}
  Let $f = \sum\limits_{k = 1}^N c_k\psi_k\in V_N$, so that 
  \begin{equation*}
  |\mathcal{V}_{\varphi}f(\lambda_j)|^2 = \sum\limits_{k = 1}^N\sum\limits_{l = 1}^N c_k\overline{c_l} \langle \psi_k,\pi(\lambda_j)\varphi\rangle \overline{\langle \psi_l,\pi(\lambda_j)\varphi\rangle}.  
  \end{equation*}
  We define the $N\times N$ rank-one matrix $T_j$ as follows:
  \begin{equation*}
    (T_j)_{kl} := \langle \psi_k,\pi(\lambda_j)\varphi\rangle \overline{\langle \psi_l,\pi(\lambda_j)\varphi\rangle}.
  \end{equation*}
  Note that $|\mathcal{V}_{\varphi}f(\lambda_j)|^2 = \langle \mathbf{c},T_j \mathbf{c}\rangle_{\mathbb{C}^N}$, where $\mathbf{c}$ is the $N\times 1$ column vector $(c_1\, c_2\, \cdots\, c_N)^{\top}$. Since each random variable $\lambda_j$ is uniformly distributed over $\Omega$, and $\psi_k$ is the $k$th eigenfunction of the time-frequency localization operator $H_{\Omega,\varphi}$, 
  the expectation of the $kl$-th entry is
  \begin{align*}
      \mathbb{E}\left((T_j)_{kl}\right) &= \lfrac{1}{|\Omega|}\iint_{\Omega}\langle \psi_k,\pi(z)\varphi\rangle \overline{\langle \psi_l,\pi(z)\varphi\rangle}\,dz = \lfrac{1}{|\Omega|}\langle  H_{\Omega,\varphi}\psi_k,\psi_l\rangle\\
      &= \lfrac{1}{|\Omega|}\alpha_{k}\delta_{kl}\quad k,l = 1,\ldots,N,
  \end{align*}
  where $\delta_{kl}$ is Kronecker's delta. The expectation of $T_j$ is the diagonal matrix
  \begin{equation}\label{eq:ExpTdiag}
      \mathbb{E}(T_j) = \lfrac{1}{|\Omega|}\operatorname{diag}(\alpha_k)=:\lfrac{1}{|\Omega|}\Delta.
  \end{equation}
  Now, the expression inside the left-hand side of \eqref{eq:PropProb} can be rewritten as
  \begin{align}
      \inf_{f\in V_N,\|f\|_2 = 1}\lfrac{1}{r} & \sum_{j = 1}^r\left(|\mathcal{V}_{\varphi}f(\lambda_j)|^2-\lfrac{1}{|\Omega|}\langle  H_{\Omega,\varphi}f,f \rangle\right)\notag\\
	&=\inf_{\|\mathbf{c}\| = 1}\lfrac{1}{r} \sum_{j = 1}^r(\langle \mathbf{c},T_j \mathbf{c}\rangle_{\mathbb{C}^N} - \langle \mathbf{c},\mathbb{E}(T_j)\mathbf{c}\rangle_{\mathbb{C}^N})\notag\\
	&=\alpha_{\operatorname{min}}\left(\lfrac{1}{r} \sum_{j = 1}^r(T_j-\mathbb{E}(T_j))\right),\label{eq:inftoeigenvalue}
  \end{align}
  where $\alpha_{\text{min}}(U)$ denotes the smallest eigenvalue of a self-adjoint matrix $U$.
  
It follows from Theorem \ref{thm:Tropp}, taking $t = r\nu/|\Omega|$, that
\begin{equation*}
  \mathbb{P}\left(\alpha_{\text{min}}\bigg(\sum_{j = 1}^r(T_j-\mathbb{E}(T_j)\bigg)\leq -\lfrac{\nu r}{|\Omega|}\right)\leq N\exp\left(-\lfrac{\nu^2 r^2|\Omega|^{-2}}{|\Omega|^{-1}r+|\Omega|^{-1}\nu r/3}\right).
\end{equation*}
Together with \eqref{eq:inftoeigenvalue}, we obtain the conclusion of the proposition.
\end{proof}

In the next lemma, we observe a relation between the lower sampling inequality for the space $V_N$ to that for functions that are $(\varepsilon,\varphi)$-concentrated in $\Omega$.

\begin{lemma}\label{lem:LowerBoundSampIneqSubsp}
    Let $N\in\mathbb{N}$ and $\alpha_N\geq\gamma\geq\alpha_{N+1}$. Let $\varphi\in\So(\mathbb{R})$, with $\|\varphi\|_2 = 1$ and $\Lambda_{\Omega}=\{\lambda_r\}_{j = 1}^r$ a finite relatively separated set of points in $\Omega$. If the inequality 
    \begin{equation}\label{eq:SampTFLocSubspLowerBd}
	\mfrac{1}{r}\sum_{j = 1}^r|\cV_{\varphi}p(\lambda_j)|^2\geq \lfrac{\langle  H_{\Omega,\varphi}p,p\rangle -\nu\|p\|_2^2}{|\Omega|},
    \end{equation}
    where $\nu\geq 0$, holds for all $p\in V_N$, then the inequality
    \begin{equation}\label{eq:SampTFLocFuncLowerBd}
	\sum_{j = 1}^r|\cV_{\varphi}f(\lambda_j)|^2\geq A\|f\|_2^2
    \end{equation}
    holds for all $f$ that are $(\varepsilon,\varphi)$-concentrated in $\Omega$ with constant
    \begin{equation*}
	A = \lfrac{r}{|\Omega|}\left(\gamma-\lfrac{\gamma\varepsilon}{1-\gamma}-\nu\right)-2B\sqrt{\lfrac{\varepsilon}{1-\gamma}},
    \end{equation*}
    where $B$ is a constant dependent on the covering index 
    \begin{equation*}
      N_0 = \sup_{m\in\mathbb{Z}^2}\#(\Lambda\cap Q_1(m))
    \end{equation*}
    and the window function $\varphi$.
\end{lemma}

\begin{remark}\noindent\phantom{-}\\ \vspace{-12pt}
  \begin{enumerate}
    \item[(a)] For $A>0$, we need $r\geq |\Omega|\left(\mfrac{2B\sqrt{\frac{\varepsilon}{1-\gamma}}}{\gamma-\frac{\gamma\,\varepsilon}{1-\gamma}-\nu}\right)$.
    
    \item[(b)] We note that if \eqref{eq:SampTFLocSubspLowerBd} holds and $\gamma>\nu$, then $\{\mathcal{P}_{V_N}\pi(\lambda_j)\varphi\}_{j = 1}^r$ is a frame for $V_N$. Indeed, if $p\in V_{V_N}$, then $\cV_{\varphi}p(\lambda_j) = \langle p,\pi(\lambda_j)\varphi\rangle = \langle p,\mathcal{P}_{V_N}\pi(\lambda_j)\varphi\rangle$ and we have
  \begin{equation*}
    \sum_{j = 1}^{r}|\cV_{\varphi}p(\lambda_j)|^2= \sum_{j = 1}^{r}|\langle p,\mathcal{P}_{V_N}\pi(\lambda_j)\varphi\rangle|^2\leq \sum_{j = 1}^{r}\|p\|_2^2\|\varphi\|_2^2 = r\|p\|_2^2.
  \end{equation*}  
  Now, since $\langle H_{\Omega,\varphi}p,p\rangle \geq \alpha_N\|p\|_2^2\geq \gamma\|p\|_2^2$, \eqref{eq:SampTFLocSubspLowerBd} together with the assumption that $\gamma-\nu>0$ gives the lower frame inequality for $\{\mathcal{P}_{V_N}\pi(\lambda_j)\varphi\}_{j = 1}^{r}$.
    \label{rem:FramePropVN}
  \end{enumerate}
\end{remark}

\begin{proof}
  Since $f = \mathcal{P}_{V_N}f + (I-\mathcal{P}_{V_N})f$, we have
  \begin{equation*}
    \left(\sum_{j = 1}^r |\cV_{\varphi}f(\lambda_j)|^2\right)^{1/2}\geq \left(\sum_{j = 1}^r |\cV_{\varphi}\mathcal{P}_{V_N}f(\lambda_j)|^2\right)^{1/2} - \left(\sum_{j = 1}^r |\cV_{\varphi}(I-\mathcal{P}_{V_N})f(\lambda_j)|^2\right)^{1/2}.
  \end{equation*}
  Squaring both sides of the inequality and applying Theorem \ref{thm:BesselBd}, where $B$ is the bound given in \eqref{eq:BesselBd} that is dependent on $N_0$, we get
  \begin{align*}
    \sum_{j = 1}^r |\cV_{\varphi}f(\lambda_j)|^2 &\geq \sum_{j = 1}^r |\cV_{\varphi}\mathcal{P}_{V_N}f(\lambda_j)|^2\\ 
    &\phantom{\text{$\leq$\, }}-2\left(\sum_{j = 1}^r |\cV_{\varphi}\mathcal{P}_{V_N}f(\lambda_j)|^2\right)^{1/2}\left(\sum_{j = 1}^r |\cV_{\varphi}(I-\mathcal{P}_{V_N})f(\lambda_j)|^2\right)^{1/2}\\
    &\phantom{\text{$\leq$\, }}+\sum_{j = 1}^r |\cV_{\varphi}(I-\mathcal{P}_{V_N})f(\lambda_j)|^2\\
    &\geq \sum_{j = 1}^r |\cV_{\varphi}\mathcal{P}_{V_N}f(\lambda_j)|^2-2B\|\mathcal{P}_{V_N}f\|_2\|(I-\mathcal{P}_{V_N})f\|_2\\
    &\geq \sum_{j = 1}^r |\cV_{\varphi}\mathcal{P}_{V_N}f(\lambda_j)|^2-2B\sqrt{\mfrac{\varepsilon}{1-\gamma}}\,\|f\|_2^2,
  \end{align*}
  where the last inequality follows from $\|\mathcal{P}_{V_N}f\|_2\leq\|f\|_2$ and Lemma \ref{lem:TFlocApproxVN}(2). By hypothesis \eqref{eq:SampTFLocSubspLowerBd} and Lemma \ref{lem:TFlocApproxVN}, we obtain
  \begin{align*}
    \sum_{j = 1}^r |\cV_{\varphi}f(\lambda_j)|^2 &\geq \sum_{j = 1}^r |\cV_{\varphi}\mathcal{P}_{V_N}f(\lambda_j)|^2-2B\sqrt{\mfrac{\varepsilon}{1-\gamma}}\,\|f\|_2^2\\
    &\geq \mfrac{r}{|\Omega|}\langle  H_{\Omega,\varphi}\mathcal{P}_{V_N}f,\mathcal{P}_{V_N}f\rangle - \mfrac{r\,\nu}{|\Omega|}\|\mathcal{P}_{V_N}f\|_2^2-2B\sqrt{\mfrac{\varepsilon}{1-\gamma}}\,\|f\|_2^2\\
    &\geq \mfrac{r}{|\Omega|}\gamma\left(1-\mfrac{\varepsilon}{1-\gamma}\right)\|f\|_2^2-\mfrac{r\,\nu}{|\Omega|}\|f\|_2^2-2B\sqrt{\mfrac{\varepsilon}{1-\gamma}}\,\|f\|_2^2.
  \end{align*}
  So we can take $A$ as
  \begin{equation*}
	A = \lfrac{r}{|\Omega|}\left(\gamma-\lfrac{\gamma\varepsilon}{1-\gamma}-\nu\right)-2B\sqrt{\lfrac{\varepsilon}{1-\gamma}}.
  \end{equation*}
\end{proof}

For the succeeding results, let $\Omega$ be a compact set in $\mathbb{R}^2$ that would need at most $|\Omega|+\epsilon_1$ cubes $Q_1(m)$, with $\epsilon_1\geq 0$, to cover it.

\begin{lemma}\label{lem:ProbPtsOmega}
  Let $\Lambda_{\Omega}=\{\lambda_j\}_{j=1}^r$ be a finite sequence of independent and identically distributed random variables that are uniformly distributed in $\Omega$. Let $a>|\Omega|^{-1}$. Then
  \begin{equation*}
    \mathbb{P}(N_0>ar)\leq (|\Omega|+\epsilon_1)\exp\Big(-r\big(a\ln(a|\Omega|)-(a-|\Omega|^{-1}) \big)\Big).
  \end{equation*}
\end{lemma}

\begin{proof}
  If $N_0>ar$, then for at least one $m$, $Q_1(m)$ must contain at least $ar$ points from $\Lambda_{\Omega}$. So we have
  \begin{equation}\label{eq:ProbN0ar}
    \mathbb{P}(N_0>ar)\leq (|\Omega|+\epsilon_1)\sup_{m\in\mathbb{Z}^2}\mathbb{P}(\#(\Lambda_{\Omega}\cap Q_1(m))>ar).
  \end{equation}
  We fix $m\in\mathbb{Z}^2$. For any $b>0$, it follows from Chebyshev's inequality that
  \begin{align*}
    \mathbb{P}(\#(\Lambda_{\Omega}\cap Q_1(m))>ar) &= \mathbb{P}\Big(\sum_{j = 1}^r \chi_{Q_1(m)}(\lambda_j)>ar\Big)\\
    &\leq e^{-bar}\,\mathbb{E}\exp\Big(b\sum_{j = 1}^r \chi_{Q_1(m)}(\lambda_j)\Big).
  \end{align*}
  Since the $\lambda$'s are uniformly distributed over $\Lambda_{\Omega}$, it follows that $\chi_{Q_1(m)}(\lambda_j) = 1$ with probability at most $|\Omega|^{-1}$ and otherwise is $0$. And by the independence,  
  \begin{align*}
    \mathbb{P}(\#(\Lambda_{\Omega}\cap Q_1(m))>ar) &\leq e^{-bar}\,\prod_{j = 1}^r \mathbb{E}\exp(b\chi_{Q_1(m)}(\lambda_j))\\
    &\leq e^{-bar}\,((1-|\Omega|^{-1})+e^b|\Omega|^{-1})^r = e^{-bar}\,(1+(e^b-1)|\Omega|^{-1})^r\\
    &\leq e^{-bar}\,(\exp\big((e^b-1)|\Omega|^{-1}\big))^r.
  \end{align*}
  We choose $b = \ln(a|\Omega|)$ that optimizes the last term, which becomes
  \begin{equation*}
    \exp\Big(-r\big(a\ln(a|\Omega|)-(a-|\Omega|^{-1}) \big)\Big).
  \end{equation*}
  Substituting this expression in \eqref{eq:ProbN0ar} gives the desired result.
\end{proof}

We now combine the result in Theorem \ref{thm:ProbSampSubsp} with the estimates obtained in Lemma \ref{lem:LowerBoundSampIneqSubsp} and Lemma \ref{lem:ProbPtsOmega}, and choose appropriate values of the parameters $\varepsilon$ and $\nu$ to prove the next theorem. We take $\gamma = 1/2$ so that $N$ is around $|\Omega|$, say $N = |\Omega| + \epsilon_{2}$, by \eqref{eq:EigEst}. From the bound $B$ in \eqref{eq:BesselBd}, we take $N_1(\Lambda) = N_0$ and we let $C_{\varphi} = B/N_0$.

\begin{theorem}\label{thm:RandSampIneq}
  Let $\Lambda_{\Omega}=\{\lambda_j\}_{j\in\mathbb{N}}$ be a sequence of identically distributed random variables that are uniformly distributed in $\Omega$, and let $\varphi$ be a window function in $\So(\mathbb{R})$ with $\|\varphi\|_2 = 1$. Suppose
  \begin{equation*}
    \varepsilon<\mfrac{1}{4(1+6\sqrt{2} C_{\varphi})^2}  \qquad\text{ and }\qquad \nu< \mfrac{1}{2}-(1+6\sqrt{2}C_{\varphi})\sqrt{\varepsilon}.
  \end{equation*}
  If we let
  \begin{equation*}
    A = \mfrac{r}{|\Omega|}\Big(\mfrac{1}{2}-\varepsilon-\nu-6\sqrt{2}C_{\varphi}\sqrt{\varepsilon}\Big),
  \end{equation*}
  then the sampling inequality
  \begin{equation}\label{eq:RandSampIneq}
    A\|f\|_2^2\leq \sum_{j=1}^r|\cV_{\varphi}f(\lambda_j)|^2\leq r\|f\|_2^2,
  \end{equation}
  for all $(\varepsilon,\varphi)$-concentrated functions, holds with probability at least
  \begin{equation}\label{eq:ProbSampIneq}
    1-(|\Omega| + \epsilon_{2})\exp\bigg(-\lfrac{\nu^2r}{|\Omega|(1+\nu/3)}\bigg)-(|\Omega|+\epsilon_1)\exp\Big(-\mfrac{r}{|\Omega|}(3\ln3 -2)\Big).
  \end{equation}
\end{theorem}

\begin{proof} 
  Since $|\cV_{\varphi}f(\lambda_j)| = |\langle f,\pi(\lambda_j)\varphi\rangle|\leq \|f\|_2$, the right-hand side of \eqref{eq:RandSampIneq} follows immediately. We take  $a = 3|\Omega|^{-1}$. Let
  \begin{equation*}
    V_1 = \bigg\{\underset{f\in V_N,\,\|f\|_2 = 1}{\text{inf}}\mfrac{1}{r}\sum_{j = 1}^r \Big(|\mathcal{V}_{\varphi}f(\lambda_j)|^2-\mfrac{1}{|\Omega|}\langle H_{\Omega,\varphi}f,f\rangle\Big)\leq -\lfrac{\nu}{|\Omega|} \bigg\}
  \end{equation*}
  and let
  \begin{equation*}
    V_2 = \{N_0>ar\}.
  \end{equation*}
  It follows from Theorem \ref{thm:ProbSampSubsp} and Lemma \ref{lem:ProbPtsOmega} that the probability of $(V_1\cup V_2)^c$ is bounded below by \eqref{eq:ProbSampIneq}. And by Lemma \ref{lem:LowerBoundSampIneqSubsp}, we have that
  \begin{equation*}
    \sum_{j=1}^r|\cV_{\varphi}f(\lambda_j)|^2 \geq A\|f\|_2^2
  \end{equation*}
  for all $(\varepsilon,\varphi)$-concentrated functions $f$ such that $(V_1\cup V_2)^c$ holds. With $N_0=3|\Omega|^{-1}$, the lower bound in \eqref{eq:SampTFLocFuncLowerBd} becomes $A = \mfrac{r}{|\Omega|}\Big(\mfrac{1}{2}-\varepsilon-\nu-6\sqrt{2}C_{\varphi}\sqrt{\varepsilon}\Big)$. The assumptions on $\varepsilon$ and $\nu$ would guarantee that $A>0$.
\end{proof}
\medskip

\begin{remark}
With $N = |\Omega| + \epsilon_{2}$ and $0<\nu<1/2-(1+6\sqrt{2}C_{\varphi})\sqrt{\varepsilon}$, if $\delta$ is given and 
\begin{equation*}
  r\geq\max\bigg\{|\Omega|\mfrac{1+\nu/3}{\nu^2}\ln\mfrac{2(|\Omega|+\epsilon_2)}{\delta},\,\mfrac{|\Omega|}{3\ln 3-2}\ln\mfrac{2(|\Omega|+\epsilon_1)}{\delta}\bigg\} = |\Omega|\mfrac{1+\nu/3}{\nu^2}\ln\mfrac{2(|\Omega| + \epsilon_{2})}{\delta},
\end{equation*}
then the probability in \eqref{eq:ProbSampIneq} will be larger than $1-\delta$.
\end{remark}

\section{Approximate Reconstruction}\label{sec:ApproxRec}

In this section, we look at the approximate reconstruction of a function $f$ that is $(\varepsilon,\varphi)$-concentrated in $\Omega$. As in the set of band-limited functions that are essentially supported in an interval considered in \cite{bagr13}, the set of $(\varepsilon,\varphi)$-concentrated functions on $\Omega$ 
is not a linear space. Indeed, consider the eigenfunction $\psi_M$ corresponding to the eigenvalue $\alpha_M>1-\varepsilon$. Let $h = \sum_{k\in\mathbb{N}}c_k\psi_k$ such that the sequence $\{c_k\}_{k\in\mathbb{N}}$ satisfies the following conditions:
\begin{equation*}
  0<c_M<\mfrac{1-\varepsilon}{\alpha_M},\,\,\sum_{k\in\mathbb{N}}|c_k|^2 = 1,\,\text{ and }\, \sum_{k\in\mathbb{N}}\alpha_k|c_k|^2 = 1-\eta\varepsilon,\,\,1<\eta<\mfrac{1}{\varepsilon}.
\end{equation*}
It follows that $\|h\|_2 = 1$ and $\langle  H_{\Omega,\varphi}h,h\rangle = 1-\eta\varepsilon<1-\varepsilon$ so that $h$ is not $(\varepsilon,\varphi)$-concentrated on $\Omega$.
Choose $\delta$ such that $0<\delta\leq \mfrac{2c_M(\alpha_M-(1-\varepsilon))}{\varepsilon(\eta-1)}$, and let $f = \psi_M+\delta h$.
We calculate 
\begin{align*}
    \langle  H_{\Omega,\varphi}f,f\rangle &= \langle  H_{\Omega,\varphi}\psi_M,\psi_M\rangle + 2\delta\operatorname{Re}\langle  H_{\Omega_\varphi}\psi_M, h\rangle +\delta^2\langle  H_{\Omega,\varphi}h,h\rangle\\
    &=\alpha_M + 2\delta \alpha_M c_M + \delta^2(1-\eta\varepsilon).
\end{align*}
It follows from the conditions above that the right-hand side of the equation is greater than $(1+2\delta c_M +\delta^2)(1-\varepsilon)$, which in turn is equal to $\|f\|_2^2(1-\varepsilon)$. So $f$ and $\psi_M$ are both $(\varepsilon,\varphi)$-concentrated on $\Omega$, but $f - \psi_M = \delta h$ is not.

Moreover, it is possible to find distinct functions that are $(\varepsilon,\varphi)$-concentrated on $\Omega$ but have the same STFT samples in $\Omega$. Given a set $\Lambda_{\Omega} = \{\lambda_j\}_{j = 1}^r\subset\Omega$, we consider $\Phi = \operatorname{span}\{\pi(\lambda)\varphi\,:\,\lambda\in\Lambda_{\Omega}\}$ and let $\phi$ be an element of the orthogonal complement $\Phi^{\bot}$ of $\Phi$, so that $\cV_{\varphi}\phi(\lambda) = \langle \phi,\pi(\lambda)\varphi\rangle = 0$ for all $\lambda\in\Lambda_{\Omega}$. If we let $f$ be $(\varepsilon,\varphi)$-concentrated on $\Omega$ with $\|f\|_2 = 1$ and $\langle H_{\Omega,\varphi}f,f\rangle>1-\varepsilon$, and we let $\tilde{f} = f + \delta\phi,\,\delta>0$, then we have $\cV_{\varphi}\tilde{f}(\lambda) = \cV_{\varphi}f(\lambda)$ and we can choose $\delta$ small enough so that $\tilde{f}$ is also $(\varepsilon,\varphi)$-concentrated on $\Omega$.
  
Nonetheless, similar to \cite[Lemma 6]{bagr13}, it is possible to approximate $f$ from the local time-frequency samples as shown in the following lemma.
\begin{lemma}
  Let $\{\lambda_k\}_{k = 1}^r$ be a finite subset of $\Omega$. Then the solution to the least square problem
  \begin{equation}
    p_{\operatorname{opt}} = \underset{p\in V_N}{\operatorname{arg\,min}}\left\{\sum_{j = 1}^r|\cV_{\varphi}f(\lambda_j)-\cV_{\varphi}p(\lambda_j)|^2\right\}\label{eq:poptleastsqaure}
  \end{equation}
  satisfies the error estimate
  \begin{equation}\label{eq:ErrorEst}
    \sum_{j = 1}^r|\cV_{\varphi}f(\lambda_j)-\cV_{\varphi}p_{\operatorname{opt}}(\lambda_j)|^2\leq B\lfrac{\varepsilon}{1-\gamma}\|f\|_2^2
  \end{equation}
  for all $f$ that are $(\varepsilon,\varphi)$-concentrated in $\Omega$.
\end{lemma}
\begin{proof}
  The result follows from Theorem \ref{thm:BesselBd} and \eqref{lem:TFlocApproxVN-2}:
  \begin{align*}
    \sum_{j = 1}^r|\cV_{\varphi}f(\lambda_j)-\cV_{\varphi}p_{\operatorname{opt}}(\lambda_j)|^2 &\leq \sum_{j = 1}^r|\cV_{\varphi}f(\lambda_j)-\cV_{\varphi}\mathcal{P}_{V_N}f(\lambda_j)|^2\\
    &\leq B\|(I-\mathcal{P}_{V_N})f\|_2^2\\
    &\leq B\lfrac{\varepsilon}{1-\gamma}\|f\|_2^2.
  \end{align*}
\end{proof}

We illustrate the approximate reconstruction from the local time-frequency samples in the finite discrete setting ($\mathbb{C}^L,\,L = 480$). The experiment was done in MATLAB and the code can be downloaded from the following link:\\ 
\url{https://drive.google.com/open?id=0BxekIvg-b--xN0FmeHU2SWs1Rjg}\,. 

In this example, $\Omega$ is a circular region of radius 120 pixels and the window function $\varphi$ is a Gaussian. We consider the time-frequency localization operator $H_{\Omega,\varphi}$ and $N = 94$ eigenfunctions corresponding to the eigenvalues of $H_{\Omega,\varphi}$ greater than $\gamma = 1/2$ that form the subspace $V_N$. We let $f_1$ and $f_2$ be functions that are $(\varepsilon_1,\varphi)$- and $(\varepsilon_2,\varphi)$-concentrated in $\Omega$, with $\varepsilon_1>\varepsilon_2$ so that $f_1$ is less concentrated in $\Omega$ than $f_2$, and consider $r = 300$ distinct sampling points $\Lambda_{\Omega} = \{\lambda_j\}_{j = 1}^r$ on the time-frequency plane. The STFT of $f_1$ and $f_2$ together with the sampling points are shown in Figure \ref{fig:f1f2TFSamples} below.

\begin{figure}[t!hp]
  \includegraphics[width=0.95\linewidth]{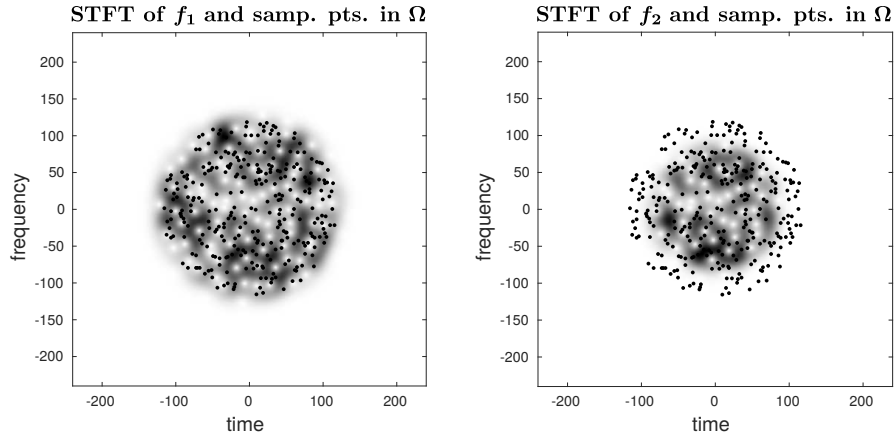}
  \caption{STFT of $f_1$ and $f_2$ and irregular sampling points in $\Omega$}\label{fig:f1f2TFSamples}
\end{figure}

We solve the least square problem in \eqref{eq:poptleastsqaure} via the conjugate gradient method applied to the corresponding normal equations and obtain $p_{\operatorname{opt}}$ in 47 iterations. The relative error 
\begin{equation*}
\frac{\sqrt{\sum_{j = 1}^r|\cV_{\varphi}f(\lambda_j)-\cV_{\varphi}p(\lambda_j)|^2}}{\|f\|_2}
\end{equation*}
for each case is computed and summarized in the table below.

\begin{center}
\begin{tabular}{|c||c|c|c|}
\hline
      & $\varepsilon_i,\,i = 1,2$ & relative error & error bound in \eqref{eq:ErrorEst}\\\hline\hline
$f_1$ & $0.0335$ & $0.13298$ & $0.72491$\\\hline
$f_2$ & $1.8252\times 10^{-9}$ & $2.0717\times 10^{-7}$ & $1.6927\times 10^{-4}$ \vphantom{$10^{0^0}$} \\ \hline
\end{tabular}
\end{center}
The results illustrate \eqref{eq:ErrorEst} in the sense that for functions that are more concentrated in $\Omega$, i.e.~with smaller $\varepsilon$, better approximate solutions can be obtained by solving for $p_{\operatorname{opt}}$ in \eqref{eq:poptleastsqaure}.

Finally, we note that by Remark \ref{rem:FramePropVN}(b), if the frame property is satisfied and $f\in V_N$, then perfect reconstruction can be obtained from the local samples $\{\cV_{\varphi}f(\lambda_j)\}_{j = 1}^r$, i.e.~$p_{\operatorname{opt}} = f$. The reconstruction procedure was applied to a function $f$ in $V_N$ and the tolerance of $10^{-12}$ was attained also in $47$ iterations.

\bibliographystyle{abbrv}

\end{document}